\documentclass[12pt]{amsart}
\usepackage[osf,sc]{mathpazo}
\usepackage{amssymb}

\usepackage{geometry}
\usepackage{bm}
\usepackage{graphicx}
\usepackage{tikz-cd}
\usepackage{hyperref}
\hypersetup{
    colorlinks=true, %set true if you want colored links
    linktoc=all,     %set to all if you want both sections and subsections linked
    linkcolor=blue,
        citecolor=red,
    filecolor=black,
    urlcolor=blue	%choose some color if you want links to stand out
}
\usepackage{enumerate}
\usepackage[inline]{enumitem}
\makeatletter
\newcommand{\inlineitem}[1][]{%
\ifnum\enit@type=\tw@
    {\descriptionlabel{#1}}
  \hspace{\labelsep}%
\else
  \ifnum\enit@type=\z@
       \refstepcounter{\@listctr}\fi
    \quad\@itemlabel\hspace{\labelsep}%
\fi} \makeatother
\newcommand{\ga}{\alpha}

\newcommand{\gl}{\lambda}
\newcommand{\gm}{\mu}
\newcommand{\gn}{\nu}

\newcommand{\gr}{\rho}

\newcommand{\gf}{\phi}

\newcommand{\Gg}{\Gamma}

\newcommand{\Gl}{\Lambda}

\newcommand{\Gs}{\Sigma}

\newcommand{\Gom}{\Omega}

\newcommand{\bs}{\backslash}
\newcommand{\fs}{/}

\newcommand{\ti}{\tilde}

\newcommand{\mbb}{\mathbb}

\newcommand{\mcl}{\mathcal}

\newcommand{\ul}{\underline}
\newcommand{\ol}{\overline}

\newcommand{\us}{\underset}
\newcommand{\os}{\overset}

\newcommand{\lra}{\longrightarrow}

\newcommand{\N}{\mbb N}
\newcommand{\Z}{\mbb Z}

\newcommand{\ZZ}[1]{\Z/p^{#1}\Z}

\newcommand{\ra}{\rightarrow}

\newcommand{\es}{\emptyset}

\newcommand{\equ}[1]{%
\begin{equation*}
#1
\end{equation*}
}
\newcommand{\equa}[1]{%
\begin{equation*}
\begin{aligned}
#1
\end{aligned}
\end{equation*}
}
\newcommand{\equan}[2]{%
\begin{equation}
\label{Eq:#1}
\begin{aligned}
#2
\end{aligned}
\end{equation}
}
\DeclareMathOperator{\Det}{Det}
\DeclareMathOperator{\Hom}{Hom}

\DeclareMathOperator{\Diag}{Diag}
\theoremstyle{plain}
\newtheorem{theorem}{Theorem}[section]

\newtheorem{lemma}[theorem]{Lemma}

\newtheorem{conj}[theorem]{Conjecture}
\newtheorem{ques}[theorem]{Question}

\makeatletter
\def\namedlabel#1#2{\begingroup
   \def\@currentlabel{#2}%
   \label{#1}\endgroup
}
\makeatother
\newtheorem*{thmOmega}{\bf{Theorem} $\bm{\Gom}$}

\theoremstyle{definition}
\newtheorem{defn}[theorem]{Definition}

\theoremstyle{remark}
\newtheorem{remark}[theorem]{Remark}
\newtheorem{example}[theorem]{Example}
\numberwithin{equation}{section}

\begin{document}
\title[A Combinatorial Identity Based on Abelian Groups]{A Combinatorial Identity for the \MakeLowercase{p}-Binomial Coefficient Based on Abelian Groups}
\author{C P Anil Kumar}
\address{School of Mathematics, Harish-Chandra Research Institute, HBNI, Chhatnag Road, Jhunsi, Prayagraj (Allahabad), 211 019,  India. \,\, email: {\tt akcp1728@gmail.com}}
\thanks{This work is done while the author is a Post Doctoral Fellow at Harish-Chandra Research Institute, Prayagraj(Allahabad).}
\date{\sc \today}
\begin{abstract}
For non-negative integers $k\leq n$, we prove a combinatorial identity for the $p$-binomial coefficient $\binom{n}{k}_p$ based on abelian p-groups. A purely combinatorial proof of this identity is not known. While proving this identity, for $r\in \N\cup\{0\},s\in \N$ and $p$ a prime, we present a purely combinatorial formula for the number of subgroups of $\Z^s$ of finite index $p^r$ with quotient isomorphic to the finite abelian $p$-group of type $\ul{\gl}$\ , which is a partition of $r$ into at most $s$ parts. This purely combinatorial formula is similar to that for the enumeration of subgroups of a certain type in a finite abelian $p$-group obtained by Lynne Marie Butler. As consequences, this combinatorial formula gives rise to many enumeration formulae that involve polynomials in $p$ with non-negative integer coefficients.
\end{abstract}
\maketitle
\section{\bf{Introduction}}
Many results in combinatorics have its proofs relying on other branches of mathematics such as number theory. In this article we prove a combinatorial identity for the p-binomial coefficient based on the theory of abelian p-groups and number theory. The question of determining the number of abelian p-subgroups of different types is a natural one in group theory with a relatively rich and interesting history. Some of the authors who have worked on this topic are G.~Birkhoff~\cite{Birkhoff}, S.~Delsarte~\cite{MR0025463}, P.~Hall~\cite{Hall}, R.~P.~Stanley~\cite{MR2868112}, L.~M.~Butler~\cite{MR1223236}, D.~E.~Knuth~\cite{MR0270933}, I.~G.~Macdonald~\cite{MR3443860}. 

Let $\Gl$ denote the set of all finite sequences of the form 
\equan{Partition}{\ul{\gl}=(\gl_1\geq \gl_2\geq \gl_3\geq \cdots \geq \gl_k)}
where $\gl_i,\ 1\leq i\leq k$ are positive integers. Let $\mid \ul{\gl}\mid$ denote the sum $\us{i=1}{\os{k}{\sum}}\gl_i$. Here we say $\ul{\gl}$ is a partition of $\mid \ul{\gl}\mid$.
We allow the case $k=0$, resulting in an empty sequence or empty partition, which we denote by $\es$ and $\mid \es \mid =0$.
It is a well known fact, in the theory of abelian groups that, for any prime $p$, any finite abelian $p$-group is, up to an isomorphism, of the form  
\equan{AbGroupPartition}{\mcl{A}_{\ul{\gl}}=\ZZ {\gl_1} \oplus \ZZ {\gl_2} \oplus \cdots \oplus \ZZ {\gl_k}} for a unique partition $\ul{\gl}\in \Gl$ (refer to Chapter $3$, Theorem $3.3.2$ on Page $41$ of M.~Hall Jr.~\cite{MR0414669}). Here we say the abelian $p$-group $\mcl{A}_{\ul{\gl}}$ is of type $\ul{\gl}$. The empty sequence corresponds to the trivial abelian $p$-group. For $n,k\in \N \cup\{0\}$ let
\equ{\Gl_{n,k}=\{\ul{\gl}\in \Gl\big\vert \mid \ul{\gl}\mid =n, \ul{\gl} \text{ has exactly }k\text{ parts}\}.}  If $k>n$ then $\Gl_{n,k}=\es$ and if $k=0=n$ then $\Gl_{0,0}=\{\es\}$.
For $n\in \N\cup\{0\}$ the set of partitions of $n$ is given by \equ{\Gl_n=\us{k=0}{\os{n}{\bigcup}}\Gl_{n,k}.}
Now we introduce another collection $\Gl^0\supsetneq \Gl$ of sequences of the form 
\equan{PartitionwithZero}{\ul{\gl}=(\gl_1\geq \gl_2\geq \gl_3\geq \cdots \geq \gl_k)}
where $\gl_i,\ 1\leq i\leq k$ are non-negative integers. In this sequence $\ul{\gl}$ which is also called a partition, the parts $\gl_i,\ 1\leq i\leq k$ are allowed to be zero. In Equation~\ref{Eq:PartitionwithZero} the partition $\ul{\gl}$ has exactly $k$-parts. We allow the case $k=0$, resulting in the empty partition and also the only partition which has either only zero parts, or no parts at all, is the empty partition. Here again we can associate an abelian $p$-group to a partition $\ul{\gl}\in \Gl^0$ in a similar way as in Equation~\ref{Eq:AbGroupPartition}
given by 
\equ{\mcl{A}_{\ul{\gl}}=\ZZ {\gl_1} \oplus \ZZ {\gl_2} \oplus \cdots \oplus \ZZ {\gl_k}.}
Note that $\ZZ 0$ is the trivial group and for $\gl=\es,\mcl{A}_{\es}=\{0\}$, the trivial group. Here if two partitions $\ul{\gl},\ \ul{\gm}\in \Gl^0$ are equal after ignoring the zero parts then we have $\mcl{A}_{\ul{\gl}}\cong \mcl{A}_{\ul{\gm}}$.   Conversely, if $\mcl{A}_{\ul{\gl}}\cong \mcl{A}_{\ul{\gm}}$ for $\ul{\gl},\ \ul{\gm}\in \Gl^0$ then $\ul{\gl},\ \ul{\gm}$ become equal after ignoring the zero parts. For $n,\ k\in \N\cup\{0\}$, let 
\equ{\Gl^0_{n,k}=\{\ul{\gl}\in \Gl^0\big\vert \mid \ul{\gl}\mid =n, \ \ul{\gl} \text{ has exactly }k\text{ parts}\}.}
In particular for $k\in \N\cup\{0\}, \Gl^0_{0,k}=\{\es\}$.  
For $n,\ k\in \N\cup\{0\}$, we see that there is a natural bijection between the sets \equ{\Gl^0_{n,k}\text{ and }\us{j=0}{\os{k}{\bigcup}} \Gl_{n,j}}
where each partition in the set $\us{j=0}{\os{k}{\bigcup}} \Gl_{n,j}$ is extended by zeroes to a partition in $\Gl^0_{n,k}$ which has exactly $k$-parts.

\begin{defn}
Let $\ul{\gl}\in \Gl$ be a partition as given in Equation~\ref{Eq:Partition}. Then the conjugate partition $\ul{\gl}'\in \Gl$ of the partition $\ul{\gl}$ is the partition is given by
\equ{\ul{\gl}'=(\gl_1'\geq \gl_2'\geq \gl_3'\geq \cdots \geq \gl_l')} where 
\equ{\gl_i'=\emph{Card}\{j\mid \gl_j\geq i\},\ 1 \leq i\leq l=\gl_1.}
We observe that $\gl_1'=k$. For example the conjugate of the partition $\ul{\gl}=(5\geq 4\geq 4\geq 1)$ is $\ul{\gl}'=(4\geq 3\geq 3\geq 3\geq 1)$. The conjugate of the empty partition is itself.
\end{defn}
\begin{defn}
Let $\ul{\gl}=(\gl_1\geq \cdots \geq \gl_k),\ \ul{\gm}=(\gm_1\geq \cdots \ge \gm_l)$ be two partitions in $\Gl$. We say $\ul{\gm}\subseteq \ul{\gl}$ if $l\leq k$ and $\gm_i\leq \gl_i$ for $1\leq i\leq l$. 
\end{defn}
\begin{remark}
It is a well-known fact from the theory of abelian $p$-groups that if $\mcl{A}_{\ul{\gm}}$ is a subgroup of $\mcl{A}_{\ul{\gl}}$ then $\ul{\gm}\subseteq \ul{\gl}$  (refer to Chapter $3$, Theorem $3.3.3$ on Page $42$ of M.~Hall Jr.~\cite{MR0414669}). 
\end{remark}
\begin{defn}
Let $\ul{\gm}\subseteq \ul{\gl}\in \Gl$ be two partitions.  For a subgroup $\mcl{A}_{\ul{\gm}}\subseteq \mcl{A}_{\ul{\gl}}$, the co-type of $\mcl{A}_{\ul{\gm}}$ is defined to be the type of abelian p-group $\frac{\mcl{A}_{\ul{\gl}}}{\mcl{A}_{\ul{\gm}}}$.
\end{defn}
\begin{defn}
A sequence of real numbers $\{c_0,c_1,\cdots,c_m\}$ is said to be unimodal if there exists an index $0\leq r\leq m$ such that we have 
\equ{c_0\leq c_1\leq \cdots \leq c_r\geq c_{r+1} \geq \cdots \geq c_m.} A polynomial $p(x)=\us{i=0}{\os{m}{\sum}}c_ix^i\in \mbb{R}[x]$ is said to be unimodal if its sequence of coefficients $\{c_0,c_1,\cdots,c_m\}$ is unimodal.
\end{defn}
\begin{defn}
For $n,k\in \N,k\leq n$, the $p$-binomial coefficient is defined as 	
	\equ{\binom{n}{k}_p=\frac{(p^n-1)\cdots(p^{n-k+1}-1)}{(p^k-1)\cdots (p-1)}.}
For $n\in \N\cup\{0\},\ k=0,\ \binom{n}{k}_p=1$ and for $n,\ k\in \N,\ k>n,\ \binom{n}{k}_p=0$. 
\end{defn}
Now we state a remark about the p-binomial coefficients. 
\begin{remark}
	\label{remark:pBinomCoefficient}
	For a partition of $\ul{\gl}\in \Gl$, where $\ul{\gl}$ is as given in Equation~\ref{Eq:Partition}, we can think of $\ul{\gl}$ as a tableau, an arrangement of $\mid \ul{\gl}\mid$ square boxes with $\gl_1$ square boxes juxtaposed in the first row, $\gl_2$ square boxes juxaposed in the second row and so on, with all the rows being left aligned. We note that in this way, the partition $\ul{\gl}$ fits in a rectangular grid of square boxes of size $k\times \gl_1$.	
	Now we observe that for $1\leq l\leq n,\ l,\ n\in \N$ we have \equ{\binom{n}{l}_p=\binom{n-1}{l-1}_p+p^l\binom{n-1}{l}_p.}
	Therefore for $0\leq l\leq n,\ l,\ n\in \N$
	\equan{Tableau}{\binom{n}{l}_p=\us{\os{\ul{\gl}\in \Gl}{\ul{\gl} \text{ fits in }l\times (n-l)\text{ rectangle }}}{\sum}p^{\mid \ul{\gl} \mid}.}
	If $l=0$ or $n-l=0$ then Equation~\ref{Eq:Tableau} follows by taking $\ul{\gl}$ as the empty partition. If $0<l<n$ then it follows easily because if the first column of $\ul{\gl}$ has $l$ squares, then delete that column to obtain a partition which fits inside a $l\times (n-l-1)$ rectangle of squares; otherwise $\ul{\gl}$ itself fits inside a $(l-1)\times (n-l)$ rectangle of squares. As a consequence of Equation~\ref{Eq:Tableau} we have that, the p-binomial coefficient $\binom{n}{l}_p$ is a polynomial in $p$ with non-negative integer coefficients.
\end{remark}
In~\cite{MR1223236}, L.~M.~Butler gives a nice introduction to the lattice $L_{\ul{\gl}}(p)$ of subgroups of a finite abelian $p$-group of type $\ul{\gl}$  mentioning results in the last few decades in this subject and also the following most elusive conjecture.
\begin{conj}
For a prime $p$ and a partition $\ul{\gl}\in \Gl$	
the number $\ga_{\ul{\gl}}(k,p)$ of subgroups of order $p^k$ in a finite abelian $p$-group of type $\ul{\gl}$ is a unimodal polynomial in $p$ for each $0\leq k \leq \mid \ul{\gl}\mid$.	
\end{conj}  
This conjecture is open for more than twenty-five years. Also in~\cite{MR1223236}, L.~M.~Butler gives a combinatorial proof of the following theorem.

\begin{theorem}
\label{theorem:SubgroupEnumeration}
Let $p$ be a prime and $(\gm_1\geq \gm_2\geq \cdots\geq \gm_l)=\ul{\gm}\subseteq \ul{\gl}=(\gl_1\geq \gl_2\geq \cdots \geq \gl_k)$ be two partitions in $\Gl$.
The number $\ga_{\ul{\gl}}(\ul{\gm},p)$ of abelian $p$-subgroups of type $\ul{\gm}$ in a group of type $\ul{\gl}$ is given by 
\equ{\ga_{\ul{\gl}}(\ul{\gm},p)=\us{j\geq 1}{\prod}p^{(\gl_j'-\gm_j')\gm_{j+1}'}\binom{\gl_j'-\gm_{j+1}'}{\gm_j'-\gm_{j+1}'}_p}
where $\ul{\gl}'=(k=\gl_1'\geq \gl_2'\geq \cdots\geq \gl_{\gl_1}'),\ \ul{\gm}'=(l=\gm_1'\geq \gm_2'\geq \cdots\geq \gm_{\gm_1}')$ are partitions in $\Gl$ conjugate to $\ul{\gl}$ and $\ul{\gm}$ respectively with the convention that $\gl_i'=0$ for $i>\gl_1$ and $\gm_j'=0$ for $j>\gm_1$. 
\end{theorem}
As a consequence of this result it immediately follows that $\ga_{\ul{\gl}}(\ul{\gm},p)$ and hence $\ga_{\ul{\gl}}(k,p)$ are polynomials in $p$ with non-negative integer coefficients. In fact in~\cite{MR1223236}, L.~M.~Butler also gives a combinatorial interpretation of these coefficients in Proposition 1.3.2 and Proposition 1.4.6. Actually she gives much more in these two propositions. The coefficients of $\ga_{\ul{\gl}}(\ul{\gm},p)$ are unimodal using R.~P.~Stanley~\cite{MR1110850}, Proposition 1 on Page 503 and Theorem 11 on Page 516. After factoring the maximum power of $p$ in $\ga_{\ul{\gl}}(\ul{\gm},p)$ the coefficients of the remaining polynomial factor is unimodal and symmetric.

In this article we prove a similar combinatorial formula (Theorem~\ref{theorem:SubgroupCounting}) for the number $\ga_{r,s}(\ul{\gl},p)$ of subgroups of $\Z^s$ having finite index $p^r$ whose quotients are all isomorphic to an abelian $p$-group of type $\ul{\gl}$ where  $r\in \N\cup\{0\},s\in \N$, $p$ is a prime and $\ul{\gl}$ is a partition of $r$ into at most $s$ parts, that is, $\ul{\gl}\in \us{j=0}{\os{s}{\bigcup}} \Gl_{r,j}$ which is in bijection with the set $\Gl^0_{r,s}$. This combinatorial formula in Theorem~\ref{theorem:SubgroupCounting} is the key ingredient to prove the combinatorial identity (Theorem~\ref{theorem:Identity}).
\subsection{\bf{The Combinatorial Identity}}
Now we state the combinatorial identity for the $p$-binomial coefficients. 
\begin{theorem}
\label{theorem:Identity}
Let $k\leq n$ be two non-negative integers and $P^{k+1}_{n+1}$ be the set of all partitions of $(n+1)$ whose first part is $(k+1)$. Then \equan{CombinatorialIdentity}{\binom{n}{k}_p=\us{\ul{\gl}\in P^{k+1}_{n+1}}{\sum}\bigg(\us{i\geq 1}{\prod}p^{(\gl_1-\gl_i)\gl_{i+1}}\binom{\gl_1-\gl_{i+1}}{\gl_1-\gl_i}_p\bigg)}
where $\ul{\gl}=(k+1=\gl_1\geq \gl_2\geq \cdots \geq \gl_{l})\in P^{k+1}_{n+1}$ with the convention that $\gl_i=0$ for $i>l$.
\end{theorem}
\begin{example}
~\\
\begin{itemize}	
\item For $n\in \N\cup\{0\}$ and $k=0$ we have $P^1_{n+1}=\{\ul{\gl}=(\gl_1=1\geq \gl_2=1\geq \cdots \geq \gl_{n+1}=1)\}$ and the identity is clear.
\item For $n=k\in \N\cup\{0\}$ we have $P^{k+1}_{n+1}=\{\ul{\gl}=(\gl_1=k+1)\}$.  $LHS=\binom{n}{k}_p=\binom{k}{k}_p=1=p^0\binom{k+1}{0}_p=RHS$.
\item For $k+1=n\in \N$ we have $P^{k+1}_{n+1}=\{\ul{\gl}=(\gl_1=k+1\geq \gl_2=1)\}$.  $LHS=\binom{n}{k}_p=\binom{k+1}{k}_p=p^0\binom{k}{0}_pp^0\binom{k+1}{k}_p=RHS$.
\item For $k\geq 1,\ k+2=n\in \N$. we have $P^{k+1}_{n+1}=\{\ul{\gl}=(\gl_1=k+1\geq \gl_2=2),\ \ul{\gm}=(\gm_1=k+1\geq \gm_2=1\geq \gm_3=1)\}$ has two partitions.
\equa{RHS&=p^0\binom{k-1}{0}_pp^0\binom{k+1}{k-1}_p+p^0\binom{k}{0}_pp^k\binom{k}{k}_pp^0\binom{k+1}{k}_p\\
&=\binom{k+1}{k-1}_p+p^k\binom{k+1}{k}_p=\binom{k+2}{k}_p=\binom{n}{k}_p=LHS.}
\item For $k=1,\ n=4$ we have $P^2_5=\{(2\geq2\geq 1),(2\geq 1\geq 1\geq 1)\}$.
$RHS=p^0\binom{0}{0}_pp^0\binom{1}{0}_pp^0\binom{2}{1}_p+p^0\binom{1}{0}_pp^1\binom{1}{1}_pp^1\binom{1}{1}_pp^0\binom{2}{1}_p=\binom{2}{1}_p+p^2\binom{2}{1}_p=\binom{4}{1}_p=LHS$.
\item For $k\geq 2,\ k+3=n\in \N$ we have $P^{k+1}_{n+1}=\{\ul{\gl}=(\gl_1=k+1\geq\gl_2=3),\ \ul{\gm}=(\gm_1=k+1\geq \gm_2=2\geq \gm_3=1),\ \ul{\gn}=(\gn_1=k+1\geq \gn_2=1\geq \gn_3=1\geq \gn_4=1)\}$ has three partitions.
\equa{RHS&=p^0\binom{k-2}{0}_pp^0\binom{k+1}{k-2}_p+p^0\binom{k-1}{0}_pp^{k-1}\binom{k}{k-1}_pp^0\binom{k+1}{k}_p\\
	&+p^0\binom{k}{0}_pp^k\binom{k}{k}_pp^k\binom{k}{k}_pp^0\binom{k+1}{k}_p\\	
&=\binom{k+1}{k-2}_p+p^{k-1}\binom{k}{k-1}_p\binom{k+1}{k}_p+p^{2k}\binom{k+1}{k}_p\\&=\binom{k+1}{k-2}_p+(p^{k-1}+p^k)\binom{k+1}{k-1}_p+p^{2k}\binom{k+1}{k}_p\\
&=\binom{k+2}{k-1}_p+p^k\binom{k+2}{k}_p=\binom{k+3}{k}_p=LHS.}
\end{itemize}
\end{example}
This identity follows as an immediate consequence of the following main theorem of the article and Theorem~\ref{theorem:TotalSubgroupCounting}. We give a proof of Theorem~\ref{theorem:Identity} later at the end of the article. A purely combinatorial proof of this identity is not known and the proof given in this article is based on counting certain set of finite abelian groups in two different ways.
\subsection{\bf{The Statements of the Important Theorems and an Open Question}}
~\\
We begin with the main result of the article.
\begin{thmOmega}
\namedlabel{theorem:SubgroupCounting}{$\Gom$}
Let $p$ be a prime and $s\in \N,r\in \N\cup\{0\}$. Let \equ{\ul{\gl}=(\gl_1\geq\gl_2\geq \ldots \geq \gl_s)\in \Gl^0_{r,s}\cong \us{j=0}{\os{s}{\bigcup}} \Gl_{r,j}} denote a partition of $r$ into exactly $s$ parts where the parts are allowed to be zero. 
Then the number $\ga_{r,s}(\ul{\gl},p)$ of subgroups of $\Z^s$ of finite index $p^r$ such that the quotients are isomorphic to the abelian $p$-group $\us{i=1}{\os{s}{\bigoplus}}\frac{\Z}{p^{\gl_i}\Z}$ is given by 
\equ{\us{i\geq 1}{\prod}p^{(\gl_1'-\gl_i')\gl_{i+1}'}\binom{\gl_1'-\gl_{i+1}'}{\gl_1'-\gl_i'}_p}
where $\ul{\gl}'=(s=\gl_1'\geq \gl_2'\geq \cdots \geq \gl_{\gl_1+t}')\in \Gl$ is the conjugate partition of $\ul{\gl}+t=(\gl_1+t\geq \gl_2+t\geq \cdots \geq \gl_s+t)\in \Gl$ for any choice of $t\in \N$ with the convention that $\gl_i'=0$ for $i>\gl_1+t$.
\end{thmOmega}
\begin{remark}
In main Theorem~\ref{theorem:SubgroupCounting}, if $\gl_s\neq 0$ then we can allow the choice $t=0$ also and $\ul{\gl}'$ can be taken to be the conjugate of $\ul{\gl}$, otherwise if $\gl_s=0$ then $t$ must be positive. 
\end{remark}
\begin{theorem}
\label{theorem:TotalSubgroupCounting}	
Let $p$ be a prime and $r\in \N\cup\{0\},s\in \N$. Then the number of subgroups of $\Z^s$ of finite index $p^r$ is 
\equ{\binom{r+s-1}{s-1}_p} 
\end{theorem}

As a consequence of Theorem~\ref{theorem:SubgroupCounting}, we conclude that $\ga_{r,s}(\ul{\gl},p)$ is a polynomial in $p$ with non-negative integer coefficients because of Remark~\ref{remark:pBinomCoefficient} and the coefficients of $\ga_{r,s}(\ul{\gl},p)$ are unimodal using R.~P.~Stanley~\cite{MR1110850}, Proposition 1 on Page 503 and Theorem 11 on Page 516. After factoring the maximum power of $p$ in $\ga_{r,s}(\ul{\gl},p)$ the coefficients of the remaining polynomial factor is unimodal and symmetric.
In fact we can conclude the following theorem which we will prove later in this article. 

\begin{theorem}
\label{theorem:ChainsofSubgroups}
Let $s,m$ be positive integers and $p$ a prime.
\begin{enumerate}
\item  For a given set $S=\{a_1<a_2<\cdots<a_m\}$ of $m$ non-negative integers let $\ga_s(S,p)$ be the number of chains of subgroups $A_m \subseteq A_{m-1}\subseteq \cdots \subseteq A_1\subseteq \Z^s$ such that the index $[\Z^s:A_i]=p^{a_i},1 \leq i\leq m$. Then $\ga_s(S,p)$ is a polynomial in $p$ with non-negative integer coefficients.
\item For a given finite sequence of $m$-partitions in $\Gl$ say $\ul{\gl}^{(1)}\subseteq \ul{\gl}^{(2)}\subseteq \cdots \subseteq \ul{\gl}^{(m)}$ with each partition having at most $s$ parts, let $\ga_s(\ul{\gl}^{(1)},\ul{\gl}^{(2)},\cdots,\ul{\gl}^{(m)},p)$ be the number of chains of finite index subgroups $A_m \subseteq A_{m-1}\subseteq \cdots \subseteq A_1\subseteq \Z^s$ such that $\frac{\Z^s}{A_i}$ is a finite abelian $p$-group of type $\ul{\gl}^{(i)},1\leq i\leq m$. Then $\ga_s(\ul{\gl}^{(1)},\ul{\gl}^{(2)},\cdots,\ul{\gl}^{(m)},p)$ is a polynomial in $p$ with non-negative integer coefficients with coefficients being unimodal. 
\end{enumerate}	
\end{theorem}

A combinatorial interpretation of the coefficients of these polynomials $\ga_s(S,p)$ and $\ga_s(\ul{\gl}^{(1)},\ul{\gl}^{(2)},\cdots,\ul{\gl}^{(m)},p)$ is not known and it is desirable to have one such interpretation similar to the results Proposition 1.3.2 and Proposition 1.4.6 obtained by L.~M.~Butler~\cite{MR1223236}.
So we have the open question.
\begin{ques}
With notations as in Theorem~\ref{theorem:ChainsofSubgroups}, give a combinatorial interpretation of the coefficients of the polynomials $\ga_s(S,p)$ and $\ga_s(\ul{\gl}^{(1)},\ul{\gl}^{(2)},\cdots,\ul{\gl}^{(m)},p)$.
\end{ques}
\section{\bf{The Proof of the Main Theorem and its Consequences}}
In this section we prove main Theorem~\ref{theorem:SubgroupCounting}.
First we mention a theorem on integral matrices and another theorem of strong approximation type.
\begin{theorem}
\label{theorem:Equivalence}
Let $s$ be a positive integer. Consider the actions of $SL_s(\Z)$ and $SL_s(\Z)\times SL_s(\Z)$ on the space $\mcl{S}=M_{s\times s}(\Z)\cap GL^{+}_s(\mbb{Q})$ consisting of $s\times s$ integer matrices with positive determinant as follows.
\begin{enumerate}[label=(\alph*)]
\item $SL_s(\Z) \times \mcl{S} \lra \mcl{S}$ given by $(U,X)\lra UX$
\item $(SL_s(\Z)\times SL_s(\Z)) \times \mcl{S} \lra \mcl{S}$ given by $((U,V),X)\lra UXV^{-1}$.
\end{enumerate}	
 Then we have the following.
\begin{enumerate}
\item In every orbit of the action in Case $(a)$, there is a unique lower triangular matrix consisting of non-negative integer entries such that each element below a diagonal element is strictly smaller than the corresponding diagonal element above it.
\item In every orbit of the action in Case $(b)$, there is a unique diagonal matrix consisting of diagonal positive integer entries $d_s\mid d_{s-1}\mid \cdots \mid d_1$. 
\end{enumerate}
\end{theorem}
\begin{proof}
For $(1)$, we observe that this is a standard result for integral matrices regarding hermite normal form of a matrix. For the proof, we may adapt the proof given in M.~Newman~\cite{MR0340283}, Theorem II.2, Page 15 and Theorem II.3, Page 18  to the action of $SL_s(\Z)$ instead of $GL_s(\Z)$.

For $(2)$, we observe that this is also a standard result for integral matrices regarding Smith normal form of a matrix. For the proof, we may adapt the proof given in M.~Newman~\cite{MR0340283}, Theorem II., Page 26, Theorem II.10, Page 28 and Theorem II.11, Page 29 to the action of $SL_s(\Z)$ instead of $GL_s(\Z)$.
\end{proof}
Now we state a theorem on strong approximation which is needed in the proof of the main theorem.
\begin{theorem}
\label{theorem:StrongApproximation}
Let $s\in \N,1<n\in \N$. The reduction map $\Z\ra \Z/n\Z$ induces a surjective map from $SL_s(\Z)$ to $SL_s(\Z/n\Z)$.
\end{theorem}
\begin{proof} 
\begin{itemize}
\item Firstly, we observe that if $R$ and $S$ are two commutative rings with unity, $\gf:R\lra S$ is a surjective ring homomorphism and $SL_s(S)$ is generated by elementary matrices then the induced map $SL_s(R)\lra SL_s(s)$ is surjective.
\item Secondly, If $S_i,1\leq i\leq m$ are finitely many commutative rings with unity such that $SL_s(S_i),1\leq i\leq m$ are generated by elementary matrices then $SL_s(\us{i=1}{\os{m}{\prod}}S_i)$ is generated by elementary matrices because $SL_s(\us{i=1}{\os{m}{\prod}}S_i) \cong \us{i=1}{\os{m}{\prod}} SL_s(S_i)$.
\item  Thirdly, if $S$ is a commutative local ring with unity then $SL_s(S)$ is generated by elementary matrices. This follows by adapting the proof of Theorem $2.2.2$ for commutative fields in Chapter $2$ on Page $63$ of J.~Rosenberg~\cite{MR1282290} to commutative local rings with unity.	
\end{itemize}
Finally, as an application of Chinese remainder theorem we obtain that the map $SL_s(\Z)\lra SL_s(\Z/n\Z)$ is surjective.  For another proof of Theorem~\ref{theorem:StrongApproximation} the reader is referred to C.~P.~Anil Kumar~\cite{MR3887364}, Theorem 1.7, Page 338.	
\end{proof}

We now prove Theorem~\ref{theorem:TotalSubgroupCounting} here.
\begin{proof}[Proof of Theorem~\ref{theorem:TotalSubgroupCounting}]
Let $A$ be a subgroup of $\Z^s$ of index $p^r$. Then $A$ is a free $\Z$-module of rank $s$. Let $M$ be an $s\times s$ integer matrix whose rows form a basis for $A$ and $\Det(M)=p^r$. Then the right multiplication of $SL_s(Z)$ by $M$ gives rise to matrices with determinant $p^r$ whose rows form bases for the same subgroup $A$. If we have two bases for $A$, then the base change matrix is in $SL_s(\Z)$, that is, it is an integer matrix of determinant one.  Hence there is a bijection between the orbits of the left action of $SL_s(\Z)$ on the subset $M^{p^r}_{s\times s}(\Z)\subseteq M_{s\times s}(\Z)\cap GL_s^{+}(\mbb{Q})$ of matrices of determinant $p^r$ and the subgroups of $\Z^s$ of index $p^r$. So it is enough to count the orbits of this action. By Theorem~\ref{theorem:Equivalence}(1) the number of such orbits is given by 
\equan{Sum}{\us{\os{b_1+b_2+\cdots+b_s=r}{b_i\geq 0,1\leq i\leq s}}{\sum}p^{b_2+2b_3+\cdots+(s-1)b_s}.}
Now the solutions $b_1+b_2+\ldots+b_s=r,b_i\geq 0,1\leq i\leq s$ bijectively correspond to tableaux of the form $\ul{\gl}=(\gl_1=b_2+\ldots+b_s\geq \gl_2=b_3+\ldots+b_s\geq \ldots \geq \gl_{s-2}=b_{s-1}+b_s\geq \gl_{s-1}=b_s)$ which fit inside a $(s-1)\times r$ rectangle as in Equation~\ref{Eq:Tableau}, where we ignore the zero parts of $\ul{\gl}$. Here we have $\mid \gl \mid =b_2+2b_3+\ldots+(s-1)b_s$. So the sum in~\ref{Eq:Sum} is precisely the $p$-binomial coefficient \equ{\binom{r+s-1}{s-1}_p} using Remark~\ref{remark:pBinomCoefficient}. This proves Theorem~\ref{theorem:TotalSubgroupCounting}. 
\end{proof}
Now we prove an useful lemma.
\begin{lemma}
\label{lemma:PartConjPart}
	Let $\ul{\gl}=(\gl_1\geq \cdots \geq \gl_s)=(\gm_1^{\gr_1}>\gm_2^{\gr_2}>\cdots>\gm_l^{\gr_l})\in \Gl^0_{r,s}$ with $\us{i=1}{\os{s}{\sum}}\gl_i=\us{i=1}{\os{l}{\sum}}\gm_i\gr_i=r,\ \us{i=1}{\os{l}{\sum}}\gr_i=s$ where $\gm_i,\ 1\leq i\leq l$ are the distinct parts of $\ul{\gl}$ occurring with multiplicities $\gr_i,\ 1\leq i\leq l$ respectively. Let $\ul{\gl}'=(s=\gl_1'\geq \gl_2'\geq \cdots\geq \gl_{\gl_1+t}')\in \Gl$ be the conjugate of $\ul{\gl}+t=(\gl_1+t\geq \gl_2+t\geq \cdots \geq \gl_s+t)\in \Gl$ for any choice of $t\in \N$ with the convention that $\gl_j'=0$ for $j>\gl_1+t$. Then 
	\equan{PartConjPart}{\us{1\leq i<j\leq l}{\sum}(\gm_i-\gm_j-1)\gr_i\gr_j=\us{j\geq 1}{\sum}(\gl_1'-\gl_j')\gl_{j+1}'.}
\end{lemma}
We mention a remark on Lemma~\ref{lemma:PartConjPart} before proving it.
\begin{remark}
	In Lemma~\ref{lemma:PartConjPart}, the first part $\gl_1'=s$ follows as a consequence of $t>0$. The positive integer $t$ is added to $\ul{\gl}$ to make sure that $\gl_1'=s$ if $\gl_s=0$ which is necessary. If $\gl_s\neq 0$ then we can allow $t=0$ and in this case $\ul{\gl}'$ can be chosen to be the conjugate of $\ul{\gl}$ itself.
\end{remark}
Next we prove Lemma~\ref{lemma:PartConjPart}.
\begin{proof}[Proof of the Lemma]
	Since $LHS$ of Equation~\ref{Eq:PartConjPart} does not change when $\ul{\gl}$ is changed to $\ul{\gl}+1$ we will prove the lemma for all partitions $\ul{\gl}$ such that $\gl_s\neq 0$, that is, $\ul{\gl}$ is in fact in $\Gl$. 
	
	So assume $\gl_s\neq 0$ and henceforth, in this proof we fix the notation 
	$\ul{\gl}=(\gl_1\geq \cdots \geq \gl_s)=(\gm_1^{\gr_1}>\gm_2^{\gr_2}>\cdots>\gm_l^{\gr_l})\in \Gl$ and the notation for its conjugate
	$\ul{\gl}'=(s=\gl_1'\geq \gl_2'\geq \cdots\geq \gl_{\gl_1}')\in \Gl$.
	
	Since $\gl_s\neq 0$, the conjugate partition of $\ul{\gl}+t=(\gl_1+t\geq \gl_2+t \geq \cdots \geq \gl_s+t)$ is given by $(s\geq s\geq \cdots \geq s\geq \gl_2'\geq \gl_3'\geq \cdots\geq \gl_{\gl_1}')$ with $s$ occuring $t+1$ times before $\gl_2'$ for any $t\in \N\cup\{0\}$. So the $RHS$ of Equation~\ref{Eq:PartConjPart} also does not change either.
	
	Now we will prove this lemma inductively. For $s=1,\ul{\gl}=(\gl_1=1)=\ul{\gl}'\in \Gl$ the lemma holds. Suppose the lemma holds for $\ul{\gl}$ with $\gl_s\neq 0$ for some $s\in \N$ then the lemma holds for $\ul{\gl}+t\in \Gl$ for all $t\in \N$. Now we prove the lemma also holds for $\ti{\ul{\gl}}=(\gl_1\geq \gl_2\geq \cdots \geq \gl_s \geq 1)\in \Gl$ where the part $1$ is added to $\ul{\gl}$ at the end.
	There are two cases. 
	\begin{enumerate}
		\item $\gl_s\neq 1$, that is $\gl_s\geq 2$. In this case $\ti{\ul{\gl}}=(\gl_1\geq \gl_2\geq \cdots \geq \gl_s \geq 1)=(\gm_1^{\gr_1}>\gm_2^{\gr_2}>\cdots>\gm_l^{\gr_l}>\gm_{l+1}^{\gr_{l+1}})\in \Gl$ where $\gm_{l+1}=1=\gr_{l+1}$ and $\widetilde{\ul{\gl}'}=(s+1\geq \gl_2'\geq \cdots\ge \gl_{\gl_1}')$. The increment in the $LHS$ of Equation~\ref{Eq:PartConjPart} is $\us{i=1}{\os{l}{\sum}}(\gm_i-\gm_{l+1}-1)\gr_i\gr_{l+1}=r-2s$. The increment in the $RHS$ of Equation~\ref{Eq:PartConjPart} is $\us{j\geq 3}{\sum}\gl_j'=r-\gl_1'-\gl_2'=r-2s$, because $\gl_1'=\gl_2'=s$ since $\gl_s\geq 2$. So we get $LHS=RHS$ for $\ti{\ul{\gl}}$.
		\item $\gl_s=1$. In this case $\ti{\ul{\gl}}=(\gl_1\geq \gl_2\geq \cdots \geq \gl_s \geq 1)=(\gm_1^{\gr_1}>\gm_2^{\gr_2}>\cdots>\gm_l^{\gr_l+1})$ where $\gm_l=1$ and $\widetilde{\ul{\gl}'}=(s+1\geq \gl_2'\geq \cdots\geq \gl_{\gl_1}')$ with $\gl'_2=s-\gr_l\leq s$ since $\gl_s=1$. The increment in the $LHS$ of Equation~\ref{Eq:PartConjPart} is $\us{i=1}{\os{l-1}{\sum}}(\gm_i-\gm_{l}-1)\gr_i=r-2s+\gr_l$. The increment in the $RHS$ of Equation~\ref{Eq:PartConjPart} is $\us{j\geq 3}{\sum}\gl_j'=r-\gl_1'-\gl_2'=r-s-(s-\gr_l)=r-2s+\gr_l$. So again we get $LHS=RHS$ for $\ti{\ul{\gl}}$. 
	\end{enumerate}
	This proves Lemma~\ref{lemma:PartConjPart}.
\end{proof}

Here we prove main Theorem~\ref{theorem:SubgroupCounting}.
\begin{proof}[Proof of the main theorem]
Let $A$ be a subgroup of $\Z^s$ of finite index $p^r$. Let $M\in M^{p^r}_{s\times s}(\Z)\subseteq M_{s\times s}(\Z)\cap GL_s^{+}(\mbb{Q})$ with $\Det(M)=p^r$ such that the rows of $M$ form a basis for $A$. Then the quotient $\frac{\Z^s}{A}$ is isomorphic to the abelian $p$-group $\mcl{A}_{\ul{\gl}}=\us{i=1}{\os{s}{\oplus}}\frac{\Z}{p^{\gl_i}\Z}$ where $\ul{\gl}\in \Gl^0_{r,s}$, if and only if, there are matrices $U,V\in SL_s(\Z)$ such that $UMV^{-1}=\Diag(p^{\gl_1},p^{\gl_2},\cdots,p^{\gl_s})$ using Theorem~\ref{theorem:Equivalence}(2) as this is a consequence of reducing the matrix $M$ into its Smith normal form. Here we note that for a subgroup $A\subseteq \Z^s$ of rank $s$ with basis as rows of a matrix $M$ which is expressed in terms of the standard basis of $\Z^s$, the rows of the matrix $UM$ gives another basis of $A$ expressed in terms of the standard basis of $\Z^s$ and the rows of the matrix $MV^{-1}$ represents the same basis of $A$ which is expressed in terms of another basis of $\Z^s$ instead of the standard basis of $\Z^s$. So the rows of $UMV^{-1}$ indeed gives a basis of the subgroup $A$ which is expressed in terms of some basis of $\Z^s$. 

So the abelian group $A$ has a basis which are the rows of the matrix \equ{UM=\Diag(p^{\gl_1},p^{\gl_2},\cdots,p^{\gl_s})V}
which is expressed in terms of standard basis of $\Z^s$.
Another matrix $\Diag(p^{\gl_1},p^{\gl_2},$ $\cdots,p^{\gl_s})W$ with $W\in SL_s(\Z)$ gives rise to a basis of its rows for the same subgroup $A$ expressed in terms of the standard basis of $\Z^s$ if and only if there exists an $X\in SL_s(\Z)$ such that \equ{X\Diag(p^{\gl_1},p^{\gl_2},\cdots,p^{\gl_s})V=\Diag(p^{\gl_1},p^{\gl_2},\cdots,p^{\gl_s})W.}
Here $X\in SL_s(\Z)$ represents the matrix of base change from rows of $\Diag(p^{\gl_1},p^{\gl_2},$ $\cdots,p^{\gl_s})V$ to rows of $\Diag(p^{\gl_1},p^{\gl_2},\cdots,p^{\gl_s})W$ in the subgroup $A$.
So we have \equ{\Diag(p^{-\gl_1},p^{-\gl_2},\cdots,p^{-\gl_s})X\Diag(p^{\gl_1},p^{\gl_2},\cdots,p^{\gl_s})=WV^{-1}\in SL_s(\Z).}
This happens if and only if $Y=WV^{-1}\in SL_s(\Z)$ has the additional property that $p^{\gl_j-\gl_i}\mid Y_{ij},\ 1\leq j<i\leq s$ by a straight forward calculation on the divisibility conditions. 	So let \equ{G_{\ul{\gl}}=\{Y\in SL_s(\Z)\mid p^{\gl_j-\gl_i}\mid Y_{ij},\ 1\leq j<i\leq s\}.}
We immediately see that $G_{\ul{\gl}}$ is a subgroup of $SL_s(\Z)$.  This is because an element $Y_0\in SL_s(\Z)$ is in $G_{\ul{\gl}}$ if and only if there exists $X_0\in SL_s(\Z)$ such that 
\equ{Y_0=\Diag(p^{-\gl_1},p^{-\gl_2},\cdots,p^{-\gl_s})X_0\Diag(p^{\gl_1},p^{\gl_2},\cdots,p^{\gl_s}).}
 Since $WV^{-1}\in G_{\ul{\gl}}$ we have $W\in G_{\ul{\gl}}V$ or the two right cosets are equal, that is, $G_{\ul{\gl}}W=G_{\ul{\gl}}V$. Conversely if, for $W,V\in SL_s(\Z),
G_{\ul{\gl}}W=G_{\ul{\gl}}V$, then the rows of the matrix $\Diag(p^{\gl_1},p^{\gl_2},\cdots,p^{\gl_s})W$ and the rows of the matrix 
$\Diag(p^{\gl_1},p^{\gl_2},\cdots,$ $p^{\gl_s})V$ form bases for the same subgroup $A\subseteq \Z^s$ such that $\frac{\Z^s}{A}$ is an abelian $p$-group isomorphic to $\mcl{A}_{\ul{\gl}}$.
Consequently the space of right cosets of the subgroup $G_{\ul{\gl}}$ in $SL_s(\Z)$ is in bijection with the set of subgroups $A$ of $\Z^s$ such that $\frac{\Z^s}{A}$ is a finite abelian p-group isomorphic to $\mcl{A}_{\ul{\gl}}$.

Now let us enumerate the coset space $SL_s(\Z)\fs G_{\ul{\gl}}$. Consider the congruence subgroup $\Gg_{\gl_1}=\{Y\in SL_s(\Z)\mid Y_{ii}\equiv 1\mod p^{\gl_1},\ 1\leq i\leq s,\ Y_{ij}\equiv 0\mod p^{\gl_1},\ 1\leq i\neq j\leq s\}$ of level $\gl_1$. Then we observe that $\Gg_{\gl_1}\subseteq G_{\ul{\gl}}\subseteq SL_s(\Z)$. Moreover the reduction map 
$SL_s(\Z)\ra SL_s\big(\frac{\Z}{p^{\gl_1}\Z}\big)$ is a surjective map (by Theorem~\ref{theorem:StrongApproximation}) whose kernel is exactly $\Gg_{\gl_1}$.
We have an exact sequence 
\equ{1\lra \Gg_{\gl_1} \lra SL_s(\Z) \lra SL_s\big(\frac{\Z}{p^{\gl_1}\Z}\big) \lra 1.}
Let $\ol{G}_{\ul{\gl}}=\frac{G_{\ul{\gl}}}{\Gg_{\gl_1}}\subseteq SL_s\big(\frac{\Z}{p^{\gl_1}\Z}\big)$. We obtain that 
\equ{SL_s(\Z)\fs G_{\ul{\gl}} \cong SL_s\big(\frac{\Z}{p^{\gl_1}\Z}\big)\fs \ol{G}_{\ul{\gl}}.}
Hence we conclude that the coset space is a finite set. So let us enumerate the coset space $SL_s\big(\frac{\Z}{p^{\gl_1}\Z}\big)\fs \ol{G}_{\ul{\gl}}$.
For this purpose we rewrite the partition $\ul{\gl}=(\gl_1\geq \cdots \geq \gl_s)\in \Gl^0_{r,s}$ in a different way which is useful. Let 
\equ{\ul{\gl}=(\gl_1\geq \cdots \geq \gl_s)=(\gm_1^{\gr_1}>\gm_2^{\gr_2}>\cdots>\gm_l^{\gr_l})\in \Gl^0_{r,s}.}
Here $\gm_i,\ 1\leq i\leq l$ are the distinct parts of $\ul{\gl}$ with multiplicities $\gr_i,\ 1\leq i\leq l$ respectively. Also $\gm_l=0$ if $\gl_s=0$.
So we have $\us{i=1}{\os{l}{\sum}}\gm_i\gr_i=r,\ \us{i=1}{\os{l}{\sum}}\gr_i=s$.
Due to the divisibility conditions on the lower triangular entries of a matrix in $G_{\ul{\gl}}$, enumeration of the group $\ol{G}_{\ul{\gl}}$ is the same as enumerating matrices of the form
\equ{\begin{pmatrix}
A^{\gr_1\times \gr_1}_{11} & A^{\gr_1\times \gr_2}_{12} & \cdots &A^{\gr_1\times \gr_{l-1}}_{1(l-1)} & A^{\gr_1\times \gr_l}_{1l}\\ \\
p^{\gm_1-\gm_2}A^{\gr_2\times \gr_1}_{21} & A^{\gr_2\times \gr_2}_{22} & \cdots &A^{\gr_2\times \gr_{l-1}}_{2(l-1)} & A^{\gr_2\times \gr_l}_{2l}\\ \\
\vdots & \vdots & \ddots & \vdots & \vdots\\ \\
p^{\gm_1-\gm_{l-1}}A^{\gr_{l-1}\times \gr_1}_{(l-1)1} & p^{\gm_2-\gm_{l-1}}A^{\gr_{l-1}\times \gr_2}_{(l-1)2} & \cdots &A^{\gr_{l-1}\times \gr_{l-1}}_{(l-1)(l-1)} & A^{\gr_{l-1}\times \gr_l}_{(l-1)l}\\ \\
p^{\gm_1-\gm_l}A^{\gr_l\times \gr_1}_{l1} & p^{\gm_2-\gm_l}A^{\gr_l\times \gr_2}_{l2} & \cdots &p^{\gm_{l-1}-\gm_l}A^{\gr_l\times \gr_{l-1}}_{l(l-1)} & A^{\gr_l\times \gr_l}_{ll}\\ \\
\end{pmatrix}\in SL_s(\frac{\Z}{p^{\gm_1}\Z})
}
where $p^{\gm_j-\gm_i}A^{\gr_i\times \gr_j}_{ij}\in M_{\gr_i\times \gr_j}(\frac{\Z}{p^{\gm_1}\Z})$ for $1\leq j<i\leq s$ and $A^{\gr_i\times \gr_j}_{ij}\in M_{\gr_i\times \gr_j}(\frac{\Z}{p^{\gm_1}\Z})$ for $1\leq i \leq j\leq s$.
Since the matrix is block upper triangular $\mod p$ such a matrix is invertible if and only if each diagonal block $A^{\gr_i\times \gr_i}_{ii}$ is invertible. Moreover we have additional condition that the matrix has determinant $1\in \ZZ {\gm_1}$. 

So the cardinality of $\ol{G}_{\ul{\gl}}$ is given by 

\equa{\mid \ol{G}_{\ul{\gl}}\mid &=\frac{p^{\big(\us{1\leq i<j\leq l}{\sum}(\gm_1-\gm_i+\gm_j)\gr_i\gr_j+\us{1\leq j<i\leq l}{\sum}\gm_1\gr_i\gr_j\big)}\us{i=1}{\os{l}{\prod}}\mid GL_{\gr_i}(\Z/p^{\gm_1}\Z)\mid}{\mid GL_1(\ZZ {\gm_1}) \mid}\\
&=\frac{p^{\big(\us{1\leq i<j\leq l}{\sum}(2\gm_1-\gm_i+\gm_j)\gr_i\gr_j\big)}\us{i=1}{\os{l}{\prod}}p^{(\gm_1-1)\gr_i^2}\mid GL_{\gr_i}(\Z/p\Z)\mid}{\mid GL_1(\ZZ {\gm_1}) \mid}}
\equa{	&=\frac{p^{\big(\us{1\leq i<j\leq l}{\sum}(2\gm_1-\gm_i+\gm_j)\gr_i\gr_j+\us{i=1}{\os{l}{\sum}}(\gm_1-1)\gr_i^2\big)}\us{i=1}{\os{l}{\prod}}\mid GL_{\gr_i}(\Z/p\Z)\mid}{\mid GL_1(\ZZ {\gm_1}) \mid}.}
The cardinality of $SL_s(\ZZ {\gm_1})$ is given by 
\equ{\mid SL_s(\ZZ {\gm_1}) \mid =\frac{\mid GL_s(\ZZ {\gm_1})\mid}{\mid GL_1(\ZZ {\gm_1})\mid}=\frac{p^{(\gm_1-1)s^2}\mid GL_s(\Z/p\Z)\mid}{\mid GL_1(\ZZ {\gm_1})\mid}.}
Hence the cardinality of the coset space $SL_s\big(\frac{\Z}{p^{\gl_1}\Z}\big)\fs \ol{G}_{\ul{\gl}}$ is given by 

\equa{&\mid SL_s\big(\frac{\Z}{p^{\gl_1}\Z}\big)\fs \ol{G}_{\ul{\gl}} \mid=\frac{p^{(\gm_1-1)s^2}\mid GL_s(\Z/p\Z)\mid}{p^{\big(\us{1\leq i<j\leq l}{\sum}(2\gm_1-\gm_i+\gm_j)\gr_i\gr_j+\us{i=1}{\os{l}{\sum}}(\gm_1-1)\gr_i^2\big)}\us{i=1}{\os{l}{\prod}}\mid GL_{\gr_i}(\Z/p\Z)\mid}\\	
&=\frac{p^{\big(\us{1\leq i<j\leq l}{\sum}(\gm_i-\gm_j-2)\gr_i\gr_j\big)}\mid GL_s(\Z/p\Z) \mid}{\us{i=1}{\os{l}{\prod}}\mid GL_{\gr_i}(\Z/p\Z)\mid}
=\frac{p^{\big(\us{1\leq i<j\leq l}{\sum}(\gm_i-\gm_j-2)\gr_i\gr_j\big)}p^{\binom{s}{2}}\us{i=1}{\os{s}{\prod}}(p^i-1)}{\us{i=1}{\os{l}{\prod}}\big(p^{\binom{\gr_i}{2}}\us{j=1}{\os{\gr_i}{\prod}}(p^j-1)\big)}\\
&=p^{\big(\us{1\leq i<j\leq l}{\sum}(\gm_i-\gm_j-2)\gr_i\gr_j+\binom{s}{2}-\us{i=1}{\os{l}{\sum}}\binom{\gr_i}{2}\big)}\frac{\us{i=1}{\os{s}{\prod}}(p^i-1)}{\us{i=1}{\os{l}{\prod}}\us{j=1}{\os{\gr_i}{\prod}}(p^j-1)}\\
&=p^{\big(\us{1\leq i<j\leq l}{\sum}(\gm_i-\gm_j-1)\gr_i\gr_j\big)}\binom{s}{\gr_1}_p\binom{s-\gr_1}{\gr_2}_p\binom{s-\gr_1-\gr_2}{\gr_3}_p\ldots\binom{\gr_{l-1}+\gr_l}{\gr_{l-1}}_p\binom{\gr_l}{\gr_l}_p.}

The final expression has two parts. The first part is a power of $p$ and the second part is a product of $p$-binomial coefficients which is the $p'$-part, that is, those which correspond to primes different from $p$.
Now we use Lemma~\ref{lemma:PartConjPart} for the power of $p$ that occurs in the last expression. So the cardinality 
\equa{\mid \ol{G}_{\ul{\gl}} \bs SL_s\big(\frac{\Z}{p^{\gl_1}\Z}\big)\mid&=p^{\us{j\geq 1}{\sum}(\gl_1'-\gl_j')\gl_{j+1}'}\binom{\gl_1'-\gl_2'}{\gl_1'-\gl_1'}_p\binom{\gl_1'-\gl_3'}{\gl_1'-\gl_2'}_p\cdots\\
&=\us{j\geq 1}{\prod}p^{(\gl_1'-\gl_j')\gl_{j+1}'}\binom{\gl_1'-\gl_{j+1}'}{\gl_1'-\gl_j'}_p.}
Here in fact $\ul{\gl}'=(\gl_1'\geq \gl_2'\geq \cdots\geq \gl_{\gl_1+t}')$ can be taken to be the conjugate of $\ul{\gl}+t=(\gl_1+t\geq \gl_2+t \geq \cdots \geq \gl_s+t)$ for any $t\in\N$. We just need to make sure that $t\in \N$ if $\gl_s=0$ so that $\gl_1'=s$ and $t$ is allowed to be zero if $\gl_s\neq 0$.
This proves main Theorem~\ref{theorem:SubgroupCounting}.
\end{proof}
As a consequence of the main Theorem~\ref{theorem:SubgroupCounting} we prove Theorem~\ref{theorem:ChainsofSubgroups}. 
But first we mention a remark.
\begin{remark}[Duality in Finite Abelian $p$-groups]
	\label{remark:Duality}
	~\\
	Let $\mcl{A}_{\ul{\gl}}=\us{i=1}{\os{s}{\oplus}}\ZZ {\gl_i}$ be the finite abelian $p$-group corresponding to the partition $\ul{\gl}\in \Gl$. Define a map $B:\mcl{A}_{\ul{\gl}}\times \mcl{A}_{\ul{\gl}} \lra \ZZ {\gl_1}$ as follows.
	\equ{B\big((a_1,a_2,\cdots,a_s),(b_1,b_2,\cdots,b_s)\big)=\us{i=1}{\os{s}{\sum}}a_ib_ip^{\gl_1-\gl_i}\in \ZZ {\gl_1}.}
	For a subgroup $A\subseteq \mcl{A}_{\ul{\gl}}$ define $A^{\perp}=\{x\in \mcl{A}_{\ul{\gl}}\mid B(x,a)=0 \text{ for all }a\in A\}$. Then we observe the following.
	\begin{itemize}
		\item The type of finite abelian $p$-group $\Hom(A,\ZZ {\gl_1})$ is same as that of $A$.
		\item There is an exact sequence of abelian $p$-groups 
		\equ{0\lra A^{\perp}\lra \us{x}{\mcl{A}_{\ul{\gl}}} \lra \us{\gf_x:a\lra B(x,a)}{\Hom(A,\ZZ {\gl_1})} \lra 0}
		and hence the co-type of $A^{\perp}$ in $\mcl{A}_{\ul{\gl}}$ is the same as of the type of $A$. 
		\item $(A^{\perp})^{\perp}=A$. Hence the type of $A^{\perp}$ is same as the co-type of $A$ in $\mcl{A}_{\ul{\gl}}$.
	\end{itemize}	
\end{remark}
\begin{proof}[Proof of Theorem~\ref{theorem:ChainsofSubgroups}]
We prove $(2)$ first. Here we use both Theorem~\ref{theorem:SubgroupCounting} and Theorem~\ref{theorem:SubgroupEnumeration}. First observe that if $\frac{\Z^s}{A_m}$ is of type $\ul{\gl}^{(m)}$ and $\frac{\Z^s}{A_{m-1}}$ is of type $\ul{\gl}^{(m-1)}$ then $\frac{A_{m-1}}{A_m}\subseteq \frac{\Z^s}{A_m}$ is of co-type $\ul{\gl}^{(m-1)}$. By Remark~\ref{remark:Duality} on duality in finite abelian $p$-groups we conclude that for a fixed $A_{m}$\ , the number of such subgroups $A_{m-1}\supseteq A_{m}$ is given by $\ga_{\ul{\gl}^{(m)}}(\ul{\gl}^{(m-1)},p)$ a polynomial in $p$ using Theorem~\ref{theorem:SubgroupEnumeration} depending only on the types. Let $r=\mid\ul{\gl}^{(m)}\mid$. We inductively conclude with a similar reasoning that \equ{\ga_s(\ul{\gl}^{(1)},\ul{\gl}^{(2)},\cdots,\ul{\gl}^{(m)},p)=\ga_{r,s}(\ul{\gl}^{(m)},p)\us{i=1}{\os{m-1}{\prod}}\ga_{\ul{\gl}^{(i+1)}}(\ul{\gl}^{(i)},p)}
which is polynomial in $p$ with non-negative coefficients with coefficients being unimodal using R.~P.~Stanley~\cite{MR1110850}, Proposition 1 on Page 503 and Theorem 11 on Page 516.

We prove $(1)$. We have 
\equ{\ga_s(S,p)=\us{\us{\mid \ul{\gl}^{(i)} \mid = a_i,1\leq i\leq m}{\ul{\gl}^{(1)}\subseteq \ul{\gl}^{(2)}\subseteq \cdots \subseteq \ul{\gl}^{(m)}}}{\sum}\ga_s(\ul{\gl}^{(1)},\ul{\gl}^{(2)},\cdots,\ul{\gl}^{(m)},p)}
where the summation is over chains of partitions, with each partition having at most $s$ parts. Hence $\ga_s(S,p)$ is a polynomial in $p$ with non-negative coefficients. This completes the proof of Theorem~\ref{theorem:ChainsofSubgroups}.    
\end{proof}

Next we prove Theorem~\ref{theorem:Identity}.
\begin{proof}
Let $k=s-1,\ n=r+s-1$ for some $s\in \N,\ r\in \N\cup\{0\}$. Then $\binom{n}{k}_p=\binom{r+s-1}{s-1}_p$.	
Using Theorem~\ref{theorem:SubgroupCounting} and Theorem~\ref{theorem:TotalSubgroupCounting} we get that 

\equ{\binom{r+s-1}{s-1}_p=\us{\ul{\gl}\in \Gl^0_{r,s}}{\sum}\bigg(\us{j\geq 1}{\prod}p^{(\gl_1'-\gl_j')\gl_{j+1}'}\binom{\gl_1'-\gl_{j+1}'}{\gl_1'-\gl_j'}_p\bigg)}
where $\ul{\gl}'=(s=\gl_1'\geq \gl_2'\geq \cdots \gl_{\gl_1+1}')$ is the conjugate of $\ul{\gl}+1=(\gl_1+1\geq \gl_2+1\geq \cdots \geq \gl_s+1)$ by taking an uniform choice of $t=1$ in Theorem~\ref{theorem:SubgroupCounting}. Therefore $\ul{\gl}+1$ and $\ul{\gl}'$ are partitions of $r+s=n+1$.  There is a set bijection 
\equ{\Gl^0_{r,s} \os{\cong}{\lra} P^{s}_{r+s} \text{ given by } \ul{\gl} \lra \ul{\gl}' \text{ the conjugate of }\ul{\gl}+1.} 
So we have the following identity for two non-negative integers $k\leq n$, 
\equ{\binom{n}{k}_p=\us{\ul{\gl}\in P^{k+1}_{n+1}}{\sum}\bigg(\us{i\geq 1}{\prod}p^{(\gl_1-\gl_i)\gl_{i+1}}\binom{\gl_1-\gl_{i+1}}{\gl_1-\gl_i}_p\bigg).}
\end{proof}


\begin{thebibliography}{1}
	
\bibitem{MR3887364}
C.~P.~Anil Kumar,
\newblock {O}n the surjectivity of certain maps,
\newblock {\em Journal of Ramanujan Mathematical Society}, Vol. {\bf 33}, No. {\bf 4}, Dec. 2018, pp. 335-378, \url{http://jrms.ramanujanmathsociety.org/archieves/v33-4.html},
\url{ http://www.mathjournals.org/jrms/2018-033-004/2018-033-004-001.html}, \url{https://drive.google.com/file/d/19_HJJatvca5BuNdZ0rr4bMumtigH5oy1/preview},  \url{https://arxiv.org/pdf/1608.03728.pdf}, MR3887364	
	
	
\bibitem{Birkhoff}
G.~Birkhoff,
\newblock {S}ubgroups of {A}belian {G}roups,
\newblock{\em Proceedings of the London Mathematical Society}, Vol. {\bf 32}, Issue {\bf 1}, 1935, pp. 385-401, \url{https://doi.org/10.1112/plms/s2-38.1.385}

\bibitem{MR1223236}
L.~M.~Butler,
\newblock {S}ubgroup lattices and symmetric functions,
\newblock {\em Memoirs of the American Mathematical Society}, Vol. {\bf 112}, No. {\bf 539}, 1994, \url{http://dx.doi.org/10.1090/memo/0539}, MR1223236 

\bibitem{MR0025463}
S.~Delsarte,
\newblock {F}onctions de {M}\"{o}bius sur les groupes abeliens finis,
\newblock {\em Annals of Mathematics Second Series}, Vol. {\bf 49}, No. {\bf 3}, Jul. 1948, pp. 600-609, 10 pages, \url{https://www.jstor.org/stable/1969047}, MR0025463 

\bibitem{MR0414669}
M.~Hall Jr.,
\newblock{T}he theory of groups, {R}eprinting of the 1968 edition, 
\newblock {\em Chelsea Publishing Co, New York}, 1976, pp. xiii + 434, MR0414669 

\bibitem{Hall}
P.~Hall,
\newblock {T}he algebra of partitions,
\newblock {\em Proceedings of the Fourth Canadian Mathematical Congress}, Banff 1957, pp. 147-159

\bibitem{MR0270933}
D.~E.~Knuth,
\newblock {S}ubspaces, subsets and partitions,
\newblock {\em Journal of Combinatorial Theory, Series A}, Vol. {\bf 10}, Issue {\bf 2}, Mar. 1971, pp. 178-180, \url{https://doi.org/10.1016/0097-3165(71)90022-7}, MR0270933

\bibitem{MR3443860}
I.~G.~Macdonald,
\newblock {S}ymmetric {F}unctions and {H}all {P}olynomials,
\newblock {\em Oxford University Press, New York}, 2015, pp. xii+475, ISBN-13: 978-0-19-873912-8, \url{https://global.oup.com/academic/product/symmetric-functions-and-hall-polynomials-9780198739128?cc=in&lang=en&}, MR3443860

\bibitem{MR0340283}
M.~Newman,
\newblock {I}ntegral {M}atrices,
\newblock {\em Pure and Applied Mathematics}, Vol. 45, Academic Press, New York-London, 1972, pp. xvii+224, ISBN-13: 978-0-12-517850-1, \url{https://www.elsevier.com/books/integral-matrices/newman/978-0-12-517850-1}, MR0340283

\bibitem{MR1282290}
J.~Rosenberg,
\newblock {T}he algebraic {K}-theory and its applications,
\newblock {\em Graduate Texts in Mathematics}, Vol. {\bf 147}, Springer-Verlag, 1994, pp. x+392, ISBN: 0-387-94248-3, \url{https://doi.org/10.1007/978-1-4612-4314-4}, MR1282290

\bibitem{MR2868112}
R.~P.~Stanley,
\newblock {E}numerative {C}ombinatorics {V}ol. {I}, Second Edition, 
\newblock {\em Cambridge Studies in Advanced Mathematics, 49. Cambridge University Press, Cambridge}, June 2012, pp. xiv+626, ISBN-13: 978-1-107-60262-5, \url{https://doi.org/10.1017/CBO9781139058520}, MR2868112 

\bibitem{MR1110850}
R.~P.~Stanley,
\newblock{Log-concave and unimodal sequences in algebra, combinatorics, and geometry},
\newblock{\em Graph theory and its applications: East and West (Jinan, 1986)}, pp. 500-535,  Annals of New York Academy of Sciences, 576, New York Acad. Sci., New York, 1989, \url{ https://doi.org/10.1111/j.1749-6632.1989.tb16434.x}, MR1110850
\end{thebibliography}
\end{document}